\documentclass[reqno]{amsart}
\usepackage{amssymb,enumerate,mathrsfs}

\numberwithin{equation}{section}

\newtheorem{theorem}[equation]{Theorem}
\newtheorem{proposition}[equation]{Proposition}
\newtheorem{lemma}[equation]{Lemma}
\newtheorem{corollary}[equation]{Corollary}

\theoremstyle{definition}

\DeclareMathOperator{\bgres}{bg-res}
\DeclareMathOperator{\bgspec}{bg-spec}
\DeclareMathOperator{\Ind}{ind}
\DeclareMathOperator{\rg}{rg}
\DeclareMathOperator{\spec}{spec}
\DeclareMathOperator{\res}{res}
\DeclareMathOperator{\graph}{graph}
\DeclareMathOperator{\SA}{\mathfrak{SA}}

\def\C{\mathbb C}
\def\R{\mathbb R}

\def\Dom{\mathcal D}
\def\Sing{\mathcal E}
\def\K{\mathcal K}
\def\M{\mathcal M}

\def\L{\mathscr L}

\def\Gr{\mathrm{Gr}}

\def\Q{\mathfrak{Q}}

\def\embed{\hookrightarrow}
\def\Id{I}
\def\im{i}
\def\open#1{\smash[t]{\overset{{}_{\,\,\circ}}{#1}{}}}
\def\Wedge{\raise2ex\hbox{$\mathchar"0356$}}

\def\set#1{\{#1\}}
\def\neutral#1{}  
\def\display#1#2{\begin{minipage}[c]{#1} {#2}\end{minipage}}

\begin{document}
\title{Spectrally unstable domains}

\begin{abstract}
Let $H$ be a separable Hilbert space, $A_c:\Dom_c\subset H\to H$ a densely defined unbounded operator, bounded from below, let $\Dom_{\min}$ be the domain of the closure of $A_c$ and $\Dom_{\max}$ that of the adjoint. Assume that $\Dom_{\max}$ with the graph norm is compactly contained in $H$ and that $\Dom_{\min}$ has finite positive codimension in $\Dom_{\max}$. Then the set of domains of selfadjoint extensions of $A_c$ has the structure of a finite-dimensional manifold $\SA$ and the spectrum of each of its selfadjoint extensions is bounded from below. If $\zeta$ is strictly below the spectrum of $A$ with a given domain $\Dom_0\in \SA$, then $\zeta$ is not in the spectrum of $A$ with domain $\Dom\in \SA$ near $\Dom_0$. But $\SA$ contains elements $\Dom_0$ with the property that for every neighborhood $U$ of $\Dom_0$ and every $\zeta\in \R$ there is $\Dom\in U$ such that $\spec(A_\Dom)\cap (-\infty,\zeta)\ne \emptyset$. We characterize these ``spectrally unstable'' domains as being those satisfying a nontrivial relation with the domain of the Friedrichs extension of $A_c$.

\end{abstract}

\author{Gerardo A. Mendoza}
\address{Department of Mathematics\\ Temple University\\ Philadelphia, PA 19122}
\email{gmendoza@temple.edu}

\subjclass[2010]{Primary: 47B25, 47A10; Secondary: 47F05, 58J05, 35P05}
\keywords{Selfadjoint extensions, stability of the spectrum, manifolds with conical singularities, elliptic operators of conical type.}



\maketitle

\section{Introduction}

Throughout the paper, $H$ is a separable Hilbert space,
\begin{equation}\label{InitialDef}
A_c:\Dom_c\subset H\to H
\end{equation}
is a densely defined unbounded operator which is semibounded from below, and
\begin{equation*}
A:\Dom_{\max}\subset H\to H
\end{equation*}
is the adjoint operator, automatically an extension of the symmetric operator \eqref{InitialDef}. 

The space $\Dom_{\max}$ is a Hilbert space with the inner product 
\begin{equation}\label{GraphProduct}
(u,v)_A=(Au,Av)+(u,v),\quad u,v\in \Dom_{\max}
\end{equation}
where the inner product on the right is that of $H$. It is further assumed that the inclusion $\Dom_{\max}\embed H$ is compact and that $\Dom_{\min}$, the domain of the closure of \eqref{InitialDef} (the closure of $\Dom_c$ in $\Dom_{\max}$) has finite positive codimension in $\Dom_{\max}$. 

\smallskip
With these assumptions, all closed extensions of \eqref{InitialDef} are Fredholm and the set of domains of extensions with index $0$ can be parametrized by the elements of a compact manifold (a Grassmannian) in which the domains of the selfadjoint extensions form a smooth compact submanifold $\SA$. It is a fact that all these selfadjoint extensions have discrete spectrum bounded from below. (See Section \ref{Basics} for details.) Write $A_\Dom$ for the operator with domain $\Dom$. The assertion that 
\begin{equation*}
\display{300pt}{every $\Dom_0\in \SA$ has a neighborhood $U_0$ for which there is $C_0\in \R$ such that $\Dom\in U_0\implies \spec(A_\Dom)\subset \set{\lambda:\Re\lambda>C_0}$
}
\end{equation*}
is false. Namely, if it were to hold, then $\SA$, being compact, would admit a finite cover by open sets $U_j$ such that the spectrum of $A_\Dom$ is bounded from below by the same constant in each set $U_j$. Hence there would be an absolute lower bound for the spectra of all selfadjoint extensions, which is not true (see Lemma \ref{AllLambdaInSomeSpec} below). So in fact there is $\Dom_0\in \SA$ such that 
\begin{equation}\label{UnstableDefinition}
\display{300pt}{for every neighborhood $U$ of $\Dom_0$ and every $\zeta\in \R$ there is $\Dom\in U$ such that $\spec(A_\Dom)\cap (-\infty,\zeta)\ne \emptyset$.
}
\end{equation}

Such domains will be called spectrally unstable. The main purpose of this paper is to establish the following characterization of these domains (proof in Section~\ref{Main}): 

\begin{theorem}\label{Characterization}
Let $\Dom_F\in \SA$ be the domain of the Friedrichs extension of \eqref{InitialDef}. The element $\Dom\in \SA$ is spectrally unstable if and only if $(\Dom\cap \Dom_F)/\Dom_{\min}\ne 0$.
\end{theorem}

Viewing the problem from the perspective of the von Neumann theory \cite{vNeumann} (see \cite[Theorem X.2]{RandS}), let $\K_{\pm\im}=\ker (A_{\Dom_{\max}}\mp \im)$. With the assumptions of the first two paragraphs above, these subspaces of $H$ have the same finite dimension. Let $\Dom_0\in \SA$. The spectrum of $U_{\Dom_0}=(A_{\Dom_0}-\im)(A_{\Dom_0}+\im)^{-1}$, the Cayley transform of $A_{\Dom_0}$, consists of $1$ and a discrete subset of the circle $S^1\subset \C$. The part of the spectrum of $U_{\Dom_0}$ in $\Im\lambda<0$ accumulates at $1$, and so the fact that arbitrarily small perturbations of $\Dom_0$ to $\Dom\in \SA$ can lead to an apparently spontaneous generation of spectrum of $A_\Dom$ arbitrarily close to $-\infty$ is not surprising. What Theorem \ref{Characterization} does, is characterize those domains $\Dom_0$ for which arbitrarily small perturbations lead to spectrum of the Cayley transform spilling over from $\Im\lambda\leq 0$ to $\Im \lambda>0$ across $1$.

Note in passing that for no $\Dom\in \SA$ can the part of the spectrum of $U_\Dom$ on the semicircle in $\Im \lambda>0$ accumulate at $1$, since the spectrum of any $A_\Dom$ is bounded below by \cite[Theorem 7, pg.~217]{BirSol87}, quoted here as Theorem \ref{AllBddBelow}.

\smallskip
The key technical results are a very simple ``regularity'' result, Proposition~\ref{DeltaLikeReg}, and Theorem~\ref{FriedrichsAsLimit}, a statement concerning recovering the essential part of the domain of the Friedrichs extension as a limit of spaces associated with $\ker(A_{\Dom_{\max}}-\lambda)$. To describe these more precisely let $\Sing$ be the orthogonal complement of $\Dom_{\min}$ in $\Dom_{\max}$ and $\pi_{\max}$ the orthogonal projection on $\Sing$, all with the inner product \eqref{GraphProduct}. Domains of closed extensions of \eqref{InitialDef} correspond to the various subspaces $D\subset \Sing$ via $\Dom=D+\Dom_{\min}$, with selfadjoint extensions corresponding to the points of a submanifold $\SA$ of the Grassmannian of subspaces of $\Sing$ of a certain dimension (so it is not $\Dom_F$ that belongs to $\SA$ in Theorem \ref{Characterization}, but a certain subspace $D_F\subset \Sing$). Let $\K_\lambda=\ker (A_{\Dom_{\max}}-\lambda)$ and $K_\lambda=\pi_{\max}\K_\lambda$. Then $\lambda\mapsto K_\lambda$ is a smooth curve in $\SA$ if $\lambda$ is sufficiently negative, and $\lim_{\lambda\to -\infty} K_\lambda=D_F$. This is a consequence of the following. For any domain $\Dom=D+\Dom_{\min}$ with $D\in \SA$ and any $s\geq 0$ we define Hilbert spaces $H_\Dom^s$ using $A_\Dom$; these Sobolev-like spaces give $H_\Dom^0=H$ and $H_\Dom^1=\Dom$. For $u\in D^\perp$, the linear functional $\delta_u$ defined by $\Dom \ni v\mapsto (Av,u)-(v,Au)\in \C$ is an element of the dual space of $H_\Dom^1$, and may also be in $H^{-s}_{\Dom^\dag}$ for $0<s<1$, the dual of $H^s_\Dom$. We show that $\delta_u\notin H^{-1/2}_{\Dom^\dag}$ for $\Dom_F=D_F+\Dom_{\min}$ if $u\ne 0$.

\smallskip
Elliptic semibounded cone operators on compact manifolds $\M$ with boundary  acting on weighted $L^2$-spaces of sections of a Hermitian vector bundle $E\to\M$, 
\begin{equation*}
A:C_c^\infty(\open \M;E)\subset x^{-\nu}L^2_b(\M;E)\to x^{-\nu}L^2_b(\M;E),
\end{equation*}
have the properties stated in the first two paragraphs, see Lesch \cite[Proposition~1.3.16 and its proof]{Lesch1997}. The fine structure of the domain of the Friedrichs extension for these differential operators was given in \cite[Theorem 8.12]{GiMe2003}; the interested reader may consult these references for detailed information about such operators. The research leading to the papers \cite{GiKrMe-b,GiKrMe-a} was the motivation for looking into the instability issue. Friedrichs defined his extension in \cite{Fr1934a}. The nature of the domain in the abstract context was elucidated by Freudenthal in \cite{Fthal1936}.

\medskip
The author is grateful to T.~Krainer for suggestions that improved the manuscript and for pointing out reference \cite{BirSol87}.

\section{Domains, Selfadjointness}\label{Basics}

All closed extensions of \eqref{InitialDef} considered here will have as domain a subspace of $\Dom_{\max}$ containing $\Dom_{\min}$. Thus the domain of every closed extension of \eqref{InitialDef} is of the form
\begin{equation*}
\Dom=D+\Dom_{\min}
\end{equation*}
with $D$ a subspace of the orthogonal complement, $\Sing$, of $\Dom_{\min}$ in $\Dom_{\max}$ with respect to the inner product \eqref{GraphProduct}; $\Sing$ is finite-dimensional by hypothesis. In particular, the domain of the Friedrichs extension of \eqref{InitialDef} has the form $\Dom_F=D_F+\Dom_{\min}$ for some subspace $D_F\subset \Sing$.

The resolvent family of
\begin{equation*}
A:\Dom_F\subset H\to H
\end{equation*} 
consists of compact operators $B_F(\lambda):H\to H$, since they are also continuous as operators $H\to \Dom_F$ and the inclusion $\Dom_F\embed H$ is compact. It follows that $A$ with domain $\Dom_{\min}$ or $\Dom_{\max}$ is Fredholm, and from this and the finiteness of $\dim\Sing$, that every closed extension of \eqref{InitialDef} is Fredholm (with compact resolvent when it exists). It is easily verified that the index of $A$ with domain $\Dom= D+\Dom_{\min}$ is 
\begin{equation}\label{LeschRelIndex}
\Ind A_\Dom=\Ind A_{\Dom_{\min}}+\dim D.
\end{equation}
Since $A_{\Dom_{\min}}-\lambda\Id$ is injective for large negative $\lambda$, $\Ind A_{\Dom_{\min}}\leq 0$. And since $A_{\Dom_{\max}}-\lambda\Id$ is surjective for such $\lambda$, $\Ind A_{\Dom_{\max}}\geq 0$. From $\Ind A_{\Dom_{\max}}=\Ind A_{\Dom_{\min}}+\dim\Sing$ and $\Ind A_{\Dom_{\max}}=-\Ind A_{\Dom_{\min}}$ (because $A_{\Dom_{\max}}$ and $A_{\Dom_{\min}}$ are adjoints of each other) one derives that $\dim\Sing=2d$ with $d=-\Ind A_{\Dom_{\min}}$; this is a positive number since $\dim \Sing>0$. One can then view the set of domains of selfadjoint extensions of \eqref{InitialDef} as
\begin{equation*}
\SA=\set{D\subset \Sing:A\text{ with domain }D+\Dom_{\min}\text{ is selfadjoint}},
\end{equation*}
a subset of $\Gr_d(\Sing)$, the Grassmannian of $d$-dimensional subspaces of $\Sing$. As such, $\SA$ is a compact real analytic submanifold of dimension $d^2$ (see Proposition \ref{SAisAManifold}).

Let 
\begin{equation*}
[\cdot,\cdot\cdot]_A:\Dom_{\max}\times \Dom_{\max}\to \C
\end{equation*}
denote the skew-Hermitian form
\begin{equation*}
[u,v]_A=(Au,v)-(u,A v).
\end{equation*}
Then $[u,v]_A = 0$ if either $u$ or $v$ belongs to $\Dom_{\min}$, so
\begin{equation*}
[u,v]_A=[\pi_{\max}u,\pi_{\max}v]_A
\end{equation*}
where
\begin{equation*}
\pi_{\max}:\Dom_{\max}\to \Dom_{\max}
\end{equation*}
is the orthogonal projection on $\Sing$. The restriction of the Green form $[\cdot,\cdot\cdot]_A$ to $\Sing$ is non-degenerate because the Hilbert space adjoint of $A$ with domain $\Dom_{\max}$ is $A$ with domain $\Dom_{\min}$.

The facts collected in the following lemma can be verified directly, or following the arguments in \cite[Section 6]{GiKrMe-a}.

\begin{lemma}
We have 
\begin{equation}\label{Emax}
\Sing=\set {u\in \Dom_{\max}: Au\in \Dom_{\max}\text{ and }A^2 u=-u}. 
\end{equation}
If $u\in \Sing$, then $Au\in \Sing$, and the map 
\begin{equation}\label{A:EmaxtoEmax}
A|_\Sing:\Sing\to \Sing
\end{equation}
is an isometry with inverse $-A|_\Sing$. If $u$, $v\in\Sing$, then 
\begin{equation}\label{GreenAndInnerProduct}
[u,Av]_A=(u,v)_A.
\end{equation}
Consequently, for any subspace $D\subset \Sing$, the adjoint of
\begin{equation*}
A:D+\Dom_{\min}\subset H\to H
\end{equation*}
is
\begin{equation}\label{HilbertAdjoint}
A:A(D^\perp)+\Dom_{\min}\subset H\to H
\end{equation}
where $D^\perp$ is the orthogonal complement of $D$ in $\Sing$. Consequently
\begin{equation}\label{SA}
D\in \SA\iff A(D^\perp)=D\iff A(D)=D^\perp.
\end{equation}
and in particular, $D\in \SA\implies D^\perp\in \SA$. 
\end{lemma}

We discuss the claim about the adjoint. The combination of \eqref{Emax} and \eqref{A:EmaxtoEmax} gives $A^2|_\Sing=-\Id$, so \eqref{GreenAndInnerProduct} can also be written as
\begin{equation*}
[u,v]_A=-(u,Av)_A.
\end{equation*}
Suppose $\Dom=D+\Dom_{\min}$ with $D\subset \Sing$. The domain of the adjoint of $A_\Dom$ is $\Dom^*=D^*+\Dom_{\min}$ for some subspace $D^*\subset \Sing$. Since $A_{\Dom_{\min}}$ is symmetric, the condition that $v\in D^*$ reduces to the statement that $[u,v]_A=0$ for all $u\in D$, equivalently,
\begin{equation*}
v\in D^*\iff (u,Av)_A=0\text{ for all }u\in D.
\end{equation*}
Thus $v\in D^*\iff Av\in D^\perp$, and so $D^*=(AD)^\perp$. Also $D=(AD^*)^\perp$, so $D^\perp=AD^*$, and using $A^2=-\Id$ again we get $D^*=A(D^\perp)$, which gives the assertion in \eqref{HilbertAdjoint}.

\medskip
If $D\in \Gr_d(\Sing)$ and $T:D\to D^\perp$ is a linear map, then 
\begin{equation*}
\graph T=\set{u+Tu:u\in D}\subset \Sing
\end{equation*}
is again an element of $\Gr_d(\Sing)$. The set $U_D$ of all such elements is a neighborhood of $D$ in $\Gr_d(\Sing)$.

\begin{lemma}\label{TheSAspaces}
Suppose $D\in \SA$. Then 
\begin{equation*}
U_D\cap \SA=\set{\graph T :\text{ the map }AT:D\to D\text{ is selfadjoint}}.
\end{equation*}
Here selfadjoint means with respect to the $A$-inner product.
\end{lemma}

Since $A|_\Sing$ is unitary, if $T:D\to D^\perp$ is such that $AT:D\to D$ is selfadjoint, then also $TA:D^\perp\to D^\perp$ is selfadjoint.

\begin{proof}
Let $D\in\SA$, let $T:D\to D^\perp$ be a linear map. In view of \eqref{SA}, the condition that $\graph T \in \SA$ is that 
\begin{equation*}
(u+Tu,A(v+Tv))_A=0 \text{ for all }u,v\in D
\end{equation*}
For a general $T:D\to D^\perp$ and $u,v\in D$ we have
\begin{equation*}
(u+Tu,A(v+Tv))_A = (u,Av)_A+(u,ATv)_A+(Tu,Av)_A+(Tu,ATv)_A.
\end{equation*}
Since $D\in \SA$ and $u,v\in D$, $(u,Av)_A=0$, and since $Tu,Tv\in D^\perp$ and $D^\perp\in \SA$, also $(Tu,ATv)_A=0$. Further, since $A$ is an isometry on $\Sing$ and $A^2=-\Id$, $(Tu,Av)_A=-(ATu,v)$. Thus 
\begin{equation*}
(u+Tu,A(v+Tv))_A=(u,ATv)_A-(ATu,v)_A
\end{equation*}
so $\graph T\in \SA$ iff $AT:D\to D$ is selfadjoint with respect to the $A$-inner product. 
\end{proof}

\medskip
Thus $\SA$, as a subset of $\Gr_d(\Sing)$, is structurally simple:

\begin{proposition}[\cite{GiKrMe-a} Proposition 6.3]\label{SAisAManifold}
The set $\SA$ is a smooth real-algebraic subvariety of $\Gr_d(\Sing)$.
\end{proposition}

The dimension of the vector space of selfadjoint operators $D\to D$ (a real vector space) is $d^2$, so $\SA$ is a real submanifold of $\Gr_d(\Sing)$ of dimension $d^2$.

\begin{lemma}[\cite{GiKrMe-a} Proposition 6.4]\label{AllLambdaInSomeSpec}
Every $\lambda\in \R$ appears as eigenvalue of some selfadjoint extension of $A$. 
\end{lemma}

\begin{proof}
Let $\lambda\in \R$. If $\ker(A_{\Dom_{\min}}-\lambda)\ne 0$, then $\lambda\in \spec(A_{D+\Dom_{\min}})$ for every $D\in \SA$, so the lemma holds in this case. Suppose now that $A_{\Dom_{\min}}-\lambda$ is injective and let $\K_\lambda=\ker(A_{\Dom_{\max}}-\lambda)$. Then $\K_\lambda\cap \Dom_{\min}=0$, so $K_\lambda=\pi_{\max}\K_\lambda$ has the same dimension as $\K_\lambda$. 
The injectivity of $A_{\Dom_{\min}}-\lambda$ implies the surjectivity of its adjoint, $A_{\Dom_{\max}}-\lambda$, so the index of the latter, namely $d$, is equal to the dimension of its kernel. So $K_\lambda\in \Gr_d(\Sing)$. Let $\Dom=K_\lambda+\Dom_{\min}$. To verify that $K_\lambda\in \SA$ let $u,v\in \K_\lambda$ and $u_0, v_0\in \Dom_{\min}$ (note that $\Dom=\K_\lambda+\Dom_{\min}$). Then $[u+u_0,v+v_0]_A=[u,v]_A$ using that the Hilbert space adjoint of $A_{\Dom_{\min}}$ is $A_{\Dom_{\max}}$ and that $A_{\Dom_{\min}}$ is symmetric. So
\begin{equation*}
[u+u_0,v+v_0]_A=(u,Av)-(Au,v)=(u,\lambda v)-(\lambda u,v)=0
\end{equation*}
since $\lambda\in \R$. It follows that $A_\Dom$ is symmetric, and from this and $\Ind A_\Dom = 0$, that $A$ is selfadjoint.
\end{proof}

We end with the following fundamental fact:

\begin{theorem}\label{AllBddBelow}
Let $m$ be a lower bound of $A_c$. Every selfadjoint extension of $A_c$ is semibounded from below and the part of its spectrum in $(-\infty,m)$ is discrete with at most $d$ eigenvalues counting multiplicity. 
\end{theorem}

This is \cite[Theorem 7, pg.~217]{BirSol87}. Indeed, in view of the semiboundedness of \eqref{InitialDef}, all we need to verify is that the deficiency indices of $A_c$ are finite and equal. Since $A_c$ is semibounded from below, $A_{\Dom_{\min}}-\lambda$ is injective if $\Im\lambda\ne 0$ or $\lambda\in \R$  is sufficiently negative. For such $\lambda$, $\K_\lambda=\ker(A_{\Dom_{\max}}-\lambda)$ has constant dimension $d$, because of \eqref{LeschRelIndex} and the definition of $d$ as $-\Ind A_{\Dom_{\min}}$. In particular, the spaces $\K_\im$ and $\K_{-\im}$ have the same dimension. But these spaces are the orthogonal complements in $H$ of the ranges of $A_{\Dom_{\min}}+\im$ and $A_{\Dom_{\min}}-\im$. We note in passing that both $\K_\im$ and $\K_{-\im}$ are subspaces of $\Sing$, with $\Sing=\K_\im\oplus \K_{-\im}$. This is the decomposition of $\Sing$ into the eigenspaces of the almost complex structure of $\Sing$ determined by $A$.

\section{$\Dom$-Sobolev spaces}

Let $A:\Dom\subset H\to H$ be a selfadjoint extension of \eqref{InitialDef}, let 
\begin{equation*}
\Pi_{\Dom,\lambda}:H\to H
\end{equation*}
be the orthogonal projection on $\ker(A_\Dom-\lambda)$. Define, for arbitrary $s\geq 0$,
\begin{equation*}
H^s_{\Dom}=\set{u\in H:\sum_{\lambda\in \spec(A_\Dom)} (1+|\lambda|)^{2s} \|\Pi_{\Dom,\lambda}u\|^2<\infty}.
\end{equation*}
This is a Hilbert space with inner product
\begin{equation*}
(u,v)_s=\sum_{\lambda\in \spec(A_\Dom)} (1+|\lambda|)^{2s}(\Pi_{\Dom,\lambda}u,\Pi_{\Dom,\lambda}v).
\end{equation*}
We will write $\|\cdot \|_s$ for the norm of $H^s_{\Dom}$. We shall not make explicit the dependence on $\Dom$ of the norm or the inner product, and omit $s$ altogether when $s=0$. 

Clearly $H^{s'}_\Dom$ is densely and continuously contained in $H^s_\Dom$ if $s'>s\geq 0$. 

\begin{lemma}\label{HmNormIsANorm}
The spaces $H^1_\Dom$ and $\Dom$ are equal and the $A$-norm on $\Dom$ and the norm of $H^1_\Dom$ are equivalent. The space $\Dom_c$ is contained in $H^s_{\Dom}$ for every $0\leq s\leq 1$, and its closure in $H^1_\Dom$ is $\Dom_{\min}$.
\end{lemma}

In particular, $H^1_\Dom\ne\Dom_{\max}$ since $\Dom\ne \Dom_{\max}$. We will write $\dot H^s_\Dom$ for the closure of $\Dom_c$ in $H^s_\Dom$ ($0\leq s\leq 1$). Evidently $\dot H^1_\Dom$ is independent of $\Dom$ (despite the notation), but $\dot H^s_\Dom$ may depend on $\Dom$ if $s<1$.

\begin{proof}
Suppose $v\in H^1_\Dom$, let 
\begin{equation*}
v_n=\sum_{\lambda<n} \Pi_{\Dom,\lambda}(v)
\end{equation*}
and note that
\begin{equation*}
Av_n=\sum_{\lambda<n}\lambda \Pi_{\Dom,\lambda}(v)
\end{equation*}
Since $v\in H$, $v_n\to v$ in $H$, but since in fact $v\in H^1_\Dom$, $Av_n$ also converges in $H$. Since $A_\Dom$ is closed, $v\in \Dom$. Thus $H^1_\Dom\subset\Dom$. The opposite inclusion follows from an application of the Spectral Function Theorem. An explicit calculation gives
\begin{equation*}
\frac{1}{4}\|u\|_1^2\leq \|u\|_A^2 \leq \|u\|_1^2,\quad u\in \Dom.
\end{equation*}
That the closure of $\Dom_c$ in $H^1_\Dom$ is $\Dom_{\min}$ follows from this and that $\Dom_c\subset H^s_\Dom$ for $0\leq s\leq 1$ follows form $H^1_\Dom\subset H^s_\Dom$ for such $s$. 
\end{proof}

Let $H^{-s}_{\Dom^\dag}$ be the dual of $H^s_\Dom$ with the norm topology. Denote the pairing of $\psi\in H^{-s}_{\Dom^\dag}$ and $u\in H^s_\Dom$ by $\langle \psi,u\rangle_s$. Define $h_s^\sharp:H^s_\Dom\to H^{-s}_{\Dom^\dag}$ by setting 
\begin{equation}\label{h_s}
\langle h_s^\sharp v,u\rangle_s=(u,v)_s.
\end{equation}
The Riesz representation theorem gives that the map $h_s^\sharp$ is surjective, so invertible since it is also injective, and an antilinear isometry.  The inverse will be denoted $h_s^\flat$. 

The space $H^{-s}_{\Dom^\dag}$ is again a Hilbert space with inner product
\begin{equation*}
(\psi,\eta)_{-s}=(h_s^\flat \eta, h_s^\flat \psi)_s, \quad \psi, \eta \in H^{-s}_{\Dom^\dag}.
\end{equation*}
The Hilbert space norm of an element  of $H^{-s}_{\Dom^\dag}$ is equal its norm as linear functional $H^s_{\Dom}\to\C$. 

\medskip
Suppose $0\leq s\leq 1$, let $\dot H^{-s}_\Dom$ be the dual of $\dot H^s_\Dom$. The inclusion map 
\begin{equation*}
\iota_s:\dot H^1_\Dom\to H^s_{\Dom}
\end{equation*}
gives the dual map
\begin{equation*}
\iota_s^\dag:H^{-s}_{\Dom^\dag}\to \dot H^{-1}_\Dom.
\end{equation*}
We are interested in the elements of the kernel of these maps. 

The kernel of $\iota_s^\dag$, the annihilator  in $H^{-s}_{\Dom^\dag}$ of the closure of $\dot H^1_{\Dom}$ in $H^s_\Dom$, is isomorphic via $h_s^\flat$ to the orthogonal complement of $\dot H^s_\Dom$ in $H^s_\Dom$, so $\dim\ker \iota_s^\dag=\dim H^s_\Dom/\dot H^s_\Dom$. In particular, $\dim\ker \iota_1^\dag=d$, since by Lemma~\ref{HmNormIsANorm}, $\dot H^1_{\Dom}=\Dom_{\min}$ and $H^1_{\Dom}=D+\Dom_{\min}$.

Suppose $0\leq s<s'\leq 1$, and let $j_{s,s'}:H^{s'}_{\Dom}\embed H^s_{\Dom}$ be the inclusion map. Then $\iota_{s}=j_{s,s'}\circ \iota_{s'}$, so $\iota_{s}^\dag = \iota_{s'}^\dag\circ j_{s,s'}^\dag$. Since $j_{s,s'}$ has dense image, $j_{s,s'}^\dag$ is injective. Consequently $u\in \ker \iota_{s}^\dag$ if and only if $\iota_{s'}^\dag (j_{s,s'}^\dag(u))=0$ and we deduce that $j_{s,s'}^\dag$ restricts to an injective map $\ker \iota_s^\dag\to \ker \iota_{s'}^\dag$. Identifying $H^{-s}_{\Dom^\dag}$ with its image in $H^{-s'}_{\Dom^\dag}$ by $j_{s,s'}^\dag$ this means 
\begin{equation}\label{KerInclusions}
\ker \iota_s^\dag = H^{-s}_{\Dom^\dag}\cap \ker \iota_{s'}^\dag,\quad 0\leq s<s'.
\end{equation}
All that is left is to determine $\ker \iota_1^\dag$.

\begin{proposition}
The kernel of $\iota_1^\dag$ consists of all maps $\delta_u:H^1_\Dom\to \C$ of the form 
\begin{equation}\label{DefDeltau}
H^1_{\Dom}\ni \psi \mapsto \langle\delta_u,\psi\rangle = [\psi,u]_A \in \C.
\end{equation}
for some $u\in D^\perp$. Here, as before, $D^\perp$ is the orthogonal complement of $D$ in $\Sing$.
\end{proposition}

\begin{proof}
Let $u\in D^\perp$. The functional $\delta_u$ is clearly linear. Its continuity as a map $\delta_u:H^1_\Dom\to \C$ is an immediate consequence of the Cauchy-Schwarz inequality, the definition of the $A$-norm and the equivalence of the latter and that of $H^1_\Dom$. If $\psi\in \dot H^1_\Dom$, then $[\psi,u]_A=0$ because $\dot H^1_\Dom=\Dom_{\min}$ and $D^\perp\subset \Dom_{\max}$, so $\delta_u\in \ker \iota_1^\perp$. If $\delta_u=0$, then $(A\psi,u)-(\psi,Au)=0$ for all $\psi\in \Dom$, since $D^\perp+\Dom_{\min}$ is the domain of the adjoint of $A_\Dom$. So $u$ belongs to the domain of the adjoint of $A_\Dom$. But since $A_\Dom$ is selfadjoint, we must have $u\in\Dom$, so $u=0$. So the map
\begin{equation*}
D^\perp\ni u\mapsto \delta_u\in H^{-1}_{\Dom^\dag}
\end{equation*}
is an antilinear isomorphism into $\ker \iota_1^\dag$. The surjectivity follows from the equality of the dimensions of $D^\perp$ and $H^1_\Dom/\dot H^1_\Dom\approx D$.
\end{proof}

\section{Estimates}

For $D\in \SA$ we let $\mathcal P_{D^\perp}$ be the collection of functionals \eqref{DefDeltau}:
\begin{equation*}
\mathcal P_{D^\perp}=\set{\delta_u:u\in D^\perp}.
\end{equation*}
Because of \eqref{KerInclusions}, elements of $\mathcal P_{D^\perp}$ may have better regularity (the number $-s$) than $H^{-1}_{\Dom^\dag}$, but of course no element $\delta_u$ with $u\ne 0$ belongs to $H^{0}_{\Dom^\dag}$. The following proposition gives an upper bound for the regularity of elements in $\ker \iota_1^\dag$ in the case where $\Dom$ is the domain of the Friedrichs extension of $A$.

\begin{proposition}\label{DeltaLikeReg} 
Let $\Dom_F=D_F+\Dom_{\min}$ be the domain of the Friedrichs extension of \eqref{InitialDef}. Then $\mathcal P_{\Dom_F^\perp}\cap H^{-1/2}_{\Dom_F^\dag}=0$. 
\end{proposition}

\begin{proof}
We show that $\dot H^{1/2}_{\Dom_F}=H^{1/2}_{\Dom_F}$ (so also $\dot H^s_{\Dom_F}=H^s_{\Dom_F}$ if $0\leq s\leq 1/2$ because of~\eqref{KerInclusions}), an equality we obtain directly by following the construction of the Friedrichs extension of $A$. Let
\begin{equation*}
\Q(u,v)=(Au,v)+c(u,v),\quad u,v \in \dot H^1_\Dom
\end{equation*}
with a large enough constant $c$. The norms on $\dot H^1_{\Dom_F}$ induced by $\Q$ and that of $H^{1/2}_{\Dom_F}$ are equivalent, so the $\Q$-com\-ple\-tion of $\dot H^1_\Dom$ can be identified with $\dot H^{1/2}_{\Dom_F}$. Let
\begin{equation*}
B:H\to \dot H^{1/2}_{\Dom_F}
\end{equation*}
be the operator such that
\begin{equation*}
\Q(Bu,v)=(u,v)\text{ for all } u\in H,\ v\in \dot H^{1/2}_{\Dom_F}.
\end{equation*}
Then $B$ is injective and its image is the domain of the Friedrichs extension of $A+c\Id$, which is the same as that of $A$. That is, $\Dom_F \subset \dot H^{1/2}_{\Dom_F}$, which is to say that $H^1_{\Dom_F}\subset \dot H^{1/2}_{\Dom_F}$. Since $H^1_{\Dom_F}$ is dense in $H^{1/2}_{\Dom_F}$, $\dot H^{1/2}_{\Dom_F}$ is a dense subspace of $H^{1/2}_{\Dom_F}$. Thus $\dot H^{1/2}_{\Dom_F} = H^{1/2}_{\Dom_F}$.
\end{proof}

Returning to the case of an arbitrary domain $\Dom$ on which $A$ is selfadjoint, let $\set{\lambda_k}_{k=1}^\infty$ be the sequence of eigenvalues of $A_\Dom$ repeated according to multiplicity and in increasing order, and let $\set{\psi_k}\subset \Dom$ be an orthonormal basis of $H$ corresponding to these eigenvalues. 

The $\psi_k$ are also a complete $A$-orthogonal system for $\Dom$. Therefore, an element $u\in \Dom_{\max}$ belongs to $D^\perp$ if and only if $(u,\psi_k)_A=0$ for all $k$:
\begin{equation*}
u\in D^\perp\iff \lambda_k(Au,\psi_k)+(u,\psi_k)=0\quad\text{for all }k.
\end{equation*}
Let $u\in D^\perp$. The relations
\begin{equation*}
\left\{\begin{aligned}
\lambda_k(u,\psi_k)-(Au,\psi_k)&=\overline{\langle\delta_u,\psi_k\rangle}\\
(u,\psi_k)+\lambda_k(Au,\psi_k)&=0,
\end{aligned}
\right.
\end{equation*}
where the first identity comes from the definition of $\delta_u$ and the second is the orthogonality condition just mentioned, give
\begin{equation}\label{deltaU}
(u,\psi_k)=\lambda_k\frac{\overline{\langle\delta_u,\psi_k\rangle}}{1+\lambda_k^2},\qquad (Au,\psi_k) = -\frac{\overline{\langle\delta_u,\psi_k\rangle}}{1+\lambda_k^2}.
\end{equation}

We will now express the elements of $\mathcal P_{D^\perp}$ as a Fourier series related to the orthonormal basis $\set{\psi_k}$. Recalling the maps $h_s^\sharp:H^s_\Dom\to H^{-s}_{\Dom^\dag}$ defined in \eqref{h_s}, let $\psi_k^0=h_0^\sharp\psi_k$. Since the inclusion map $j_s:H^s_{\Dom}\embed H^0_{\Dom}$ has dense image, the dual map
\begin{equation*}
j_s^\dag:H^{0}_{\Dom^\dag}\to H^{-s}_{\Dom^\dag}
\end{equation*}
is injective with dense image. So $\psi_k^0$ can be regarded as an element of $H^{-s}_{\Dom^\dag}$ for any $s\geq 0$. From the definition of the inner product we get $(\psi_k^0,\psi_\ell^0)_{0}=\delta_{k\ell}$. For $w\in H^s_{\Dom}$ we have
\begin{equation*}
\langle j_s^\dag \psi_k^0, w\rangle_{s} = \langle  \psi_k^0, j_s w\rangle_{0}=(w,\psi_k)=\frac{(w,\psi_k)_s}{(1+|\lambda_k|)^{2s}}  = \frac{\langle h_s^\sharp\psi_k,w\rangle_s}{(1+|\lambda_k|)^{2s}}
\end{equation*}
so, using the inverse $h_s^\flat$ of $h_s^\sharp$ ,
\begin{equation*}
h_s^\flat(j_s^\dag \psi_k^0) = (1+|\lambda_k|)^{-2s}\psi_k.
\end{equation*}
In particular,
\begin{equation*}
\|j_s^\dag \psi_k^0\|_{-s}^2 = (j_s^\dag \psi_k^0,j_s^\dag \psi_\ell^0)_{-s} =  (1+|\lambda_k|)^{-2s}\delta_{k\ell}.
\end{equation*}
If $v\in H^{-s}_{\Dom^\dag}$, then
\begin{equation*}
(v,j_s^\dag \psi_k^0)_{-s}=(h_s^\flat(j_s^\dag\psi_k^0),h_s^\flat v)_s=\langle v,h_s^\flat(j_s^\dag\psi_k^0)\rangle_s = \frac{\langle v,\psi_k\rangle_s}{(1+|\lambda_k|)^{2s}}.
\end{equation*}
Thus the Fourier series representation of $v$ is
\begin{equation*}
v=\sum_k \langle v,\psi_k\rangle_{s} j_s^\dag\psi_k^0.
\end{equation*}
The norm of an element $v=\sum_k v_k \,j_s^\dag\psi_k^0\in H^{-s}_{\Dom^\dag}$ is given by
\begin{equation*}
\|v\|_{-s}^2=\sum_k (1+|\lambda_k|)^{-2s}|v_k|^2.
\end{equation*}

Suppose now $u\in D^\perp$ and $\delta_u\in H^{-s}_{\Dom^\dag}$. Then
\begin{equation*}
\langle\delta_u,\psi_k\rangle_s = (1+|\lambda_k|)^{2s}(\delta_u,j_m^\dag \psi_k^0)_{-s},
\end{equation*}
hence
\begin{equation}\label{sNormOfDeltaU}
\|\delta_u\|_{-s}^2= \sum \frac{|\langle\delta_u,\psi_k\rangle_{s}|^2}{(1+|\lambda_k|)^{2s}}.
\end{equation}
Note that $\langle\delta_u,\psi_k\rangle_{s}$ is just $\langle\delta_u,\psi_k\rangle$ since $\psi_k\in H^s_{\Dom}$ for any $0\leq s\leq 1$.

\section{The bundle of kernels}

The background spectrum of $A$, denoted $\bgspec(A)$ is the set 
\begin{equation*}
\set{\lambda\in \C:A_{\Dom_{\min}}-\lambda\text{ is not injective or }A_{\Dom_{\max}}-\lambda\text{ is not surjective}},
\end{equation*}
see \cite{GiKrMe-a}. Its complement is denoted $\bgres(A)$. The background spectrum is of interest in that it is a subset of the spectrum of every extension of $A$. 

In the present case, since $A$ is semibounded and admits an extension with compact resolvent, the set $\bgspec(A)$ is (if not empty) a discrete subset of the real line with only $+\infty$ as a possible point of accumulation, equal to
\begin{equation*}
\bgspec(A) \\= \set{\lambda\in \C:A_{\Dom_{\min}}-\lambda\text{ is not injective}}.
\end{equation*}
Indeed, if $\lambda\in \R$ then $\ker(A_{\Dom_{\min}}-\lambda)=\rg(A_{\Dom_{\max}}-\lambda)^\perp$.

For $\lambda\in \bgres(A)$ define
\begin{equation*}
\K_\lambda=\ker(A_{\Dom_{\max}}-\lambda).
\end{equation*}
Since $A_{\min}-\lambda$ is injective if $\lambda\in \bgres(A)$, formula  \eqref{LeschRelIndex} with $\Dom=\Dom_{\max}$ gives $\dim\K_\lambda=d$. For these $\lambda$, $\K_\lambda\cap \Dom_{\min}=0$. It follows that $K_\lambda=\pi_{\max}\K_\lambda$ also has dimension $d$ for each $\lambda\in \bgspec(A)$. (These spaces are the fibers of a holomorphic vector bundle over $\bgres(A)$ that extends across $\bgspec(A)$ as a holomorphic vector bundle.)

The following lemma makes explicit the relevancy of these spaces.

\begin{lemma}\label{KcapD}
Let $D\in \Gr_d(\Sing)$. The spectrum of $A$ with domain $\Dom=D+\Dom_{\min}$ is 
\begin{equation*}
\set{\lambda\in \bgspec(A):K_\lambda\cap D\ne 0}\cup\bgspec(A).
\end{equation*}
\end{lemma}

Indeed, if $\lambda\in \spec(A_\Dom)$ and $\lambda\in \bgres(A)$, then $\ker(A_\Dom-\lambda)=\Dom \cap \K_\lambda\ne 0$, and $u\in \ker(A_\Dom-\lambda)$ if and only if $\pi_{\max} u\in K_\lambda$ and $\pi_{\max} u \in D$.
\medskip

Because of the property expressed in the lemma it is of interest to have a formula for the spaces $\K_\lambda$ when $\lambda\notin\bgspec(A)$. We get one such formula with the aid of the resolvent of an arbitrary selfadjoint extension $A_\Dom$ of \eqref{InitialDef}. 

Let then $D\in \SA$, write $\pi_{D^\perp}$, $\pi_D:\Dom_{\max}\to\Dom_{\max}$ for the $A$-orthogonal projections on $D^\perp$ and $D$, respectively, and let $\pi_{\Dom}:\Dom_{\max}\to \Dom_{\max}$ be the orthogonal projection on $\Dom$ (so $\pi_{\Dom}=1-\pi_{D^\perp}$). Let $B_{\Dom}(\lambda)$ be the resolvent of $A_\Dom$. Suppose $\lambda\in \res(A_\Dom)$ and $\phi\in \K_\lambda$. Then
\begin{equation*}
\phi=\pi_{D^\perp}\phi+\pi_{\Dom}\phi
\end{equation*}
gives
\begin{equation*}
0=(A-\lambda) \pi_{D^\perp}\phi+(A-\lambda)\pi_{\Dom}\phi.
\end{equation*}
Applying $B_\Dom(\lambda)$ get
\begin{equation*}
\pi_{\Dom}\phi=-B_\Dom(\lambda)(A-\lambda) \pi_{D^\perp}\phi
\end{equation*}
since $\pi_{\Dom}\phi\in \Dom$. Thus
\begin{equation*}
\phi = \pi_{D^\perp}\phi-B_\Dom(\lambda)(A-\lambda) \pi_{D^\perp}\phi
\end{equation*}
Conversely, it is easily verified that if $u\in D^\perp$, then
\begin{equation*}
\phi_u(\lambda)=u-B_\Dom(\lambda)(A-\lambda)u
\end{equation*}
is an element of $\K_\lambda$ for each $\lambda\in \res(A_\Dom)$. Evidently, the map $D^\perp\ni u\mapsto \phi_u(\lambda)\in \K_\lambda$ is bijective and depends holomorphically on $\lambda\notin \spec(A_\Dom)$.

Using the orthonormal basis $\set{\psi_k}$ consisting of eigenfunctions of $A_\Dom$, the formula
\begin{equation*}
B_\Dom(\lambda)f=\sum_k\frac{(f,\psi_k)}{\lambda_k-\lambda}\psi_k
\end{equation*}
and the formulas \eqref{deltaU} give
\begin{equation*}
\phi_u(\lambda)= u + \sum_k \frac{(1+\lambda\lambda_k)\overline {\langle\delta_u,\psi_k\rangle}}{(1+\lambda_k^2)(\lambda_k-\lambda)}\psi_k,\quad \lambda\notin \spec(A_\Dom);
\end{equation*}
the series converges absolutely and uniformly in $H^1_\Dom$ on compact subsets of $\res(A_\Dom)$. Alternatively, again using \eqref{deltaU} in the expansion of $u$ in terms of the $\psi_k$, we have
\begin{equation}\label{PhiUbis}
\phi_u(\lambda)=\sum_k \frac{\overline {\langle\delta_u,\psi_k\rangle}}{\lambda_k-\lambda} \psi_k,\quad \lambda\notin \spec(A_\Dom).
\end{equation}
This series converges in $H^0_\Dom$ since
\begin{equation*}
\sum_k \frac{|\langle\delta_u,\psi_k\rangle|^2}{(1+\lambda_k)^2}
\end{equation*}
converges (because $\delta_u\in H^{-1}_{\Dom^\dag}$).

\section{Negativity and regularity}

We continue our discussion with the selfadjoint operator $A_\Dom$ of the previous section; so $\Dom=D+\Dom_{\min}$ with $D\in \SA$. Let $S:D^\perp\to D^\perp$ be selfadjoint with respect to the $A$-inner product, let $T=AS:D^\perp\to D$, and let 
\begin{equation*}
\graph T=\set{u+Tu:u\in D^\perp},
\end{equation*}
which by Lemma~\ref{TheSAspaces} is an element of $\SA$. Let
\begin{equation*}
\Dom_T=\graph T+\Dom_{\min}.
\end{equation*}
By Lemma \ref{KcapD}, $\lambda\in \bgres(A)$ belongs to $\spec (A_{\Dom_T})$ if and only if $\graph T\cap K_\lambda\ne 0$. In particular, $\lambda\in \res(A_\Dom)$ belongs to $\spec(A_{\Dom_T})$ if and only if there is $u\in D^\perp$, $u\ne 0$, such that 
\begin{equation*}
u-\pi_{\max}B_\Dom(\lambda)(A-\lambda)u=u+Tu,
\end{equation*}
that is, if and only if $-\pi_{\max}B_\Dom(\lambda)(A-\lambda)u=ASu$. Setting
\begin{equation*}
F_{\Dom}(\lambda)=-A\pi_{\max}B_\Dom(\lambda)(A-\lambda)|_{D^\perp},
\end{equation*}
an operator $D^\perp\to D^\perp$ we thus have
\begin{equation}\label{PartSpecD}
\lambda\in \spec(A_{\Dom_T})\cap \res(A_\Dom)\iff F_{\Dom}(\lambda)+S\text{ has nontrivial kernel}.
\end{equation}

\begin{lemma}
The map $F_{\Dom}(\lambda)$ satisfies
\begin{equation}\label{BasicSAofF}
F_{\Dom}(\lambda)^*=F_{\Dom}(\overline \lambda), \quad\lambda\in \res(A_\Dom).
\end{equation}
In addition, for any $\lambda\in \res(A_\Dom)$,
\begin{equation}\label{QF-form}
(F_{\Dom}(\lambda)u,u')_A=\sum_{k=0}^\infty \frac{\overline {\langle\delta_{u},\psi_k\rangle}\langle\delta_{u'},\psi_k\rangle} {1+\lambda_k^2}\frac{1+\lambda\lambda_k}{\lambda_k-\lambda},\quad u,u'\in D^\perp.
\end{equation}
\end{lemma}

\begin{proof}
Let $u$, $u'\in D^\perp$. Then 
\begin{multline}\label{F_B}
(F_{\Dom}(\lambda)u,u')_A
=(-A\pi_{\max}B_\Dom(\lambda)(A-\lambda)u,u')_A\\=(\pi_{\max}B_\Dom(\lambda)(A-\lambda)u,Au')_A=(B_\Dom(\lambda)(A-\lambda)u,Au')_A
\end{multline}
where the first equality is the definition of $F_\Dom(\lambda)$, the second because $A|_\Sing$ is an isometry, and the third because $\Sing\perp \Dom_{\min}$ in the $A$-inner product. Using the definition of the $A$ inner product in the last term we thus have
\begin{align*}
(F_{\Dom}(\lambda)u,u')_A
&=(AB_\Dom(\lambda)(A-\lambda)u,-u')+(B_\Dom(\lambda)(A-\lambda)u,Au')\\
&=((A-\lambda)u+\lambda B_\Dom(\lambda)(A-\lambda)u,-u')+(B_\Dom(\lambda)(A-\lambda)u,Au')\\
&=-((A-\lambda)u,u')+ (B_\Dom(\lambda)(A-\lambda)u,(A-\overline \lambda) u')
\end{align*}
Likewise,
\begin{equation*}
(u,F_{\Dom}(\overline \lambda)u')_A = -(u,(A-\overline \lambda)u')+ ((A-\lambda)u,B_\Dom(\overline \lambda)(A-\overline \lambda) u').
\end{equation*}
Then \eqref{BasicSAofF} follows from noting that $((A-\lambda)u,u')=(u,(A-\overline \lambda)u')$ because $D^\perp+\Dom_{\min}$ is a selfadjoint domain and $B_\Dom(\lambda)^*=B_\Dom(\overline \lambda)$. This proves the first assertion of the lemma.

For the second, we have
\begin{multline*}
(F_{\Dom}(\lambda)u,u')_A=(B_\Dom(\lambda)(A-\lambda)u,Au')_A=-(u-B_\Dom(\lambda)(A-\lambda)u,Au')_A\\=-(\phi_u(\lambda),Au')_A=\lambda(\phi_u(\lambda),u')-(\phi_u(\lambda),Au')
\end{multline*}
using \eqref{F_B}. Using \eqref{PhiUbis} and \eqref{deltaU} we get
\begin{equation*}
\lambda(\phi_u(\lambda),u') = \sum_{k=0}^\infty \frac{\lambda\lambda_k\overline {\langle\delta_u,\psi_k\rangle}\langle\delta_{u'},\psi_k\rangle}{(1+\lambda_k^2)(\lambda_k-\lambda)}
\end{equation*}
and
\begin{equation*}
-(\phi_u(\lambda),Au') = \sum_{k=0}^\infty \frac{\overline {\langle\delta_u,\psi_k\rangle}\langle\delta_{u'},\psi_k\rangle}{(1+\lambda_k^2)(\lambda_k-\lambda)}.
\end{equation*}
The combination of these formulas gives \eqref{QF-form}.
\end{proof}

The following proposition is the key results:

\begin{proposition}
Let $\Dom=D+\Dom_{\min}$ with $D\in \SA$, let
\begin{equation*}
D^\perp_0=\set{u\in D^\perp:\delta_u\in H^{-1/2}_{\Dom^\dag}},
\end{equation*}
let $D^\perp_1\subset D^\perp$ be complementary to $D^\perp_0$ in $D^\perp$, and let $\pi_{D^\perp_1}:D^\perp\to D^\perp$ be the orthogonal projection on $D^\perp_1$. Then for every selfadjoint operator $S:D^\perp\to D^\perp$ there is $\zeta<0$ such that $\pi_{D^\perp_1}(F_\Dom(\lambda)+S)|_{D^\perp_1}$ is negative if $\lambda<\zeta$.
\end{proposition}

\begin{proof}
Suppose that the conclusion is false. Then there is a selfadjoint operator $S:D^\perp\to D^\perp$ and a sequence $\set{\zeta_\ell}_{\ell=1}^\infty$ decreasing to $-\infty$ such that $\pi_{D^\perp_1}(F_\Dom(\zeta_\ell)+S)|_{D^\perp_1}$ has a nonnegative eigenvalue for each $\ell$. Let $u_\ell\in D^\perp_1$ be an eigenvector of $F_{\Dom}(\zeta_\ell)+S$ for such an eigenvalue, with $\|u_\ell\|_A=1$. Thus
\begin{equation*}
(F_{\Dom}(\zeta_\ell)u_\ell,u_\ell)_A+(Su_\ell,u_\ell)_A\geq 0
\end{equation*}
for all $\ell$. Passing to a subsequence, we may assume that $\set{u_\ell}_{\ell=1}^\infty$ converges to some $u\in D^\perp_1$. Using \eqref{QF-form} we have
\begin{equation*}
(Su_\ell,u_\ell)_A\geq -(F_{\Dom}(\zeta_\ell)u_\ell,u_\ell)_A = -\sum_{k=0}^\infty \frac{|\langle\delta_{u_\ell},\psi_k\rangle|^2} {1+\lambda_k^2}\frac{1+\zeta_\ell\lambda_k}{\lambda_k-\zeta_\ell}
\end{equation*}
for every $\ell$. If $k_0=\min\set{k:\lambda_k>0}$ and $k\geq k_0$, then 
\begin{equation*}
\frac{1+\zeta_\ell\lambda_k}{\lambda_k-\zeta_\ell}<0
\end{equation*}
if $\zeta_\ell<-1/\lambda_{k_0}$, so bearing in mind that the $\lambda_k$ increase monotonically with $k$,  
\begin{equation*}
\sum_{k=k_0}^\infty -\frac{|\langle\delta_{u_\ell},\psi_k\rangle|^2} {1+\lambda_k^2}\frac{1+\zeta_\ell\lambda_k}{\lambda_k-\zeta_\ell}
\end{equation*}
is a series of non-negative terms if $\ell>\ell_0$ so that $\zeta_\ell<-1/\lambda_{k_0}$ for such $\ell$. Hence
\begin{equation*}
(Su_\ell,u_\ell)_A\geq -\sum_{k=0}^N \frac{|\langle\delta_{u_\ell},\psi_k\rangle|^2} {1+\lambda_k^2}\frac{1+\zeta_\ell\lambda_k}{\lambda_k-\zeta_\ell}
\end{equation*}
for every $N\geq k_0$ and all $\ell>\ell_0$. Taking the limit as $\ell\to\infty$ gives
\begin{equation*}
(Su,u)_A \geq \sum_{k=0}^N \lambda_k\frac{|\langle\delta_{u},\psi_k\rangle|^2} {1+\lambda_k^2}
\end{equation*}
for every $N$, so
\begin{equation*}
 \lim_{N\to\infty}\sum_{k=0}^N \lambda_k\frac{|\langle\delta_{u},\psi_k\rangle|^2} {1+\lambda_k^2}\leq (Su,u)_A.
\end{equation*}
Since only finitely many $\lambda_k$ can be negative, the estimate implies that 
\begin{equation*}
\sum_{k=0}^\infty |\lambda_k|\frac{|\langle\delta_{u},\psi_k\rangle|^2} {1+\lambda_k^2}
\end{equation*}
converges. This in turn implies that the norm of $\delta_u$ as an element of $H^{-1/2}_{\Dom^\dag}$ is finite, see \eqref{sNormOfDeltaU}. So $u\in D^\perp_0$, a contradiction since $\|u\|_A=1$ and $u\in D^\perp_1\cap D^\perp_0$.
\end{proof}

In particular, if $\mathcal P_{D^\perp}\cap H^{-1/2}_{\Dom^\dag} = 0$, then for every $c>0$ there is $\zeta<0$ such that $F_{\Dom}(\lambda)+c\Id$ is negative if $\lambda<\zeta$. Thus:

\begin{corollary}\label{SingDelta}
If $\mathcal P_{D^\perp}\cap H^{-1/2}_{\Dom^\dag} = 0$, then $F_{\Dom}(\lambda)$ is invertible for every sufficiently negative $\lambda$, and $\|F_{\Dom}(\lambda)^{-1}\|_{\L(D^\perp)}\to 0$ as $\lambda\to -\infty$.
\end{corollary}

The definition of $F_\Dom(\lambda)$ gives
\begin{equation*}
K_\lambda=\set{u-AF_{\Dom}(\lambda)u:u\in D^\perp}.
\end{equation*}
Since $F_{\Dom}(\lambda)$ is invertible for every sufficiently negative $\lambda$, also
\begin{equation}\label{KlambdaVSFlambdaInverse}
K_\lambda=\set{v+F_{\Dom}(\lambda)^{-1}Av:v\in D}.
\end{equation}
Thus, if $\mathcal P_{D^\perp}\cap H^{-1/2}_{\Dom^\dag}=0$, Corollary \ref{SingDelta} and \eqref{KlambdaVSFlambdaInverse} give that $K_\lambda\to D$ as $\lambda\to-\infty$. Applied to $D=D_F$ and bearing in mind Proposition \ref{DeltaLikeReg} and that the Friedrichs extension of $A$ is bounded below, we get:

\begin{theorem}\label{FriedrichsAsLimit}
Consider the curve
\begin{equation*}
\R_-\ni\lambda\mapsto K_{\lambda}\in \Gr_{d}(\Sing).
\end{equation*}
Then $K_{\lambda}\to D_F$ as $\lambda\to -\infty$.
\end{theorem}

The limit $\lim_{\lambda\to-\infty}K_\lambda$ is of course unique. Since $K_\lambda$ is independent of its representation, we have that if in \eqref{KlambdaVSFlambdaInverse} $K_\lambda \to D$ then $D=D_F$. Consequently,

\begin{theorem}
The Friedrichs domain of $A$ is the only selfadjoint domain such that $\mathcal P_{D^\perp}\cap H^{-1/2}_{\Dom^\dag} = 0$.
\end{theorem}

\begin{proposition}\label{UnstableImpliesInVariety}
Suppose $\set{D_\ell}_{\ell=1}^\infty \subset \SA$ is a sequence converging to $D$ and there is $\set{\zeta_\ell}\subset \R$ with $\zeta_\ell\to -\infty$ as $\ell\to\infty$ such that $D_\ell\cap K_{\zeta_\ell}\ne 0$. Then $D\cap D_F\ne 0$.
\end{proposition}

\begin{proof}
For each $\ell$ pick $v_\ell\in D_\ell\cap K_{\zeta_\ell}$ with $\|v_\ell\|_A=1$. Passing to a subsequence, assume that $v_\ell\to v$ as $\ell \to \infty$. Using $\Sing=D_F\oplus D_F^\perp$ gives for each $\ell$, a unique  $w_\ell\in D_F$ such that $v_\ell=w_\ell+F_{\Dom_F}(\zeta_\ell)^{-1}Aw_\ell=v_\ell$. The continuity of projections gives that $w_\ell$ converges. Now Corollary \ref{SingDelta} applied to the Friedrichs domain gives $F_{\Dom_F}(\zeta)^{-1}Aw_\ell \to 0\text{ as }\zeta\to -\infty$. Thus $w_\ell\to v$. Since $w_\ell \in D_F$, $v\in D_F$. Now, $D_\ell=\graph T_\ell$ for a unique $T_\ell:D\to D^\perp$; the statement that $D_\ell\to D$ means that $T_\ell\to 0$. Thus $w_\ell=v_\ell+T_\ell v_\ell$ for a unique $v_\ell\in D$ and as before $v_\ell$ converges, so $w_\ell$ converges to an element of $D$ which must be $v$. Since $\|v\|_A=1$, $D\cap D_F\ne 0$.
\end{proof}

\section{Spectrally unstable domains}\label{Main}

The following, a restatement of Theorem \ref{Characterization}, is our main result. 

\begin{theorem}
Let $\Dom_F=D_F+\Dom_{\min}$ be the domain of the Friedrichs extension of $A$. The element $D\in \SA$ has the property \eqref{UnstableDefinition} if and only if $D\in \mathfrak V_{D_F}$.
\end{theorem}

We have written $\mathfrak V_{D_F}=\set{D\in \SA:D\cap D_F\ne 0}$. This is a real-algebraic subvariety of $\SA$ of codimension $1$.

\begin{proof}
If $D\in \SA$, then either $\pi_{D_F^{\perp}}|_D:D\to D_F^\perp$ is injective, or not. In the first case, $D\in U_{D_F^\perp}$, and in the second, $D\in \mathfrak V_{D_F}$. Thus
\begin{equation*}
\SA=(\SA\cap U_{D_F^\perp})\cup \mathfrak V_{D_F}
\end{equation*}
as a disjoint union. 

Proposition \ref{UnstableImpliesInVariety} gives that every element of $\SA\cap U_{D_F^\perp}$ is spectrally stable, so we only need to show that every element of $\mathfrak V_{\Dom_F}$ is spectrally unstable.

Suppose $D\in \mathfrak V_{D_F}$. We will show the existence of curves $\lambda\mapsto D_\lambda$ in $\SA$ such that $D_\lambda\to D$ as $\lambda\to -\infty$ and $D_\lambda\cap K_\lambda\ne 0$. With such a curve we have that if $U$ is a neighborhood of $D$ and $\zeta<0$, then there is $\zeta'<\zeta$ such that $D_\lambda\in U$ for every $\lambda<\zeta'$. Since $K_\lambda\cap D_\lambda\ne 0$, $\lambda$ belongs to the spectrum of $A$ with domain $\Dom_\lambda=D_\lambda+\Dom_{\min}$, which shows that $D$ is spectrally unstable.

By Corollary \ref{SingDelta} and Proposition \ref{DeltaLikeReg}, the operator $F_{\Dom_F}(\lambda):D_F^\perp\to D_F^\perp$ is invertible for every sufficiently negative $\lambda$, so
\begin{equation*}
K_{\lambda}=\set{v+F_{\Dom_F}(\lambda)^{-1}Av:v\in D_F},
\end{equation*}
see \eqref{KlambdaVSFlambdaInverse}. Let $V$ be a subspace of $D\cap D_F$, $V\ne 0$. As usual let $\pi_D$ and $\pi_{D^\perp}$ be the orthogonal projections on $D$ and $D^\perp$. If $v\in V$, then 
\begin{align*}
v+F_{\Dom_F}(\lambda)^{-1} Av
&=\pi_D(v+F_{\Dom_F}(\lambda)^{-1} Av)+\pi_{D^\perp}(v+F_{\Dom_F}(\lambda)^{-1} Av)\\
&=(v+\pi_DF_{\Dom_F}(\lambda)^{-1} Av)+\pi_{D^\perp} F_{\Dom_F}(\lambda)^{-1} Av.
\end{align*}
Let
\begin{equation*}
V_\lambda=\set{v+\pi_DF_{\Dom_F}(\lambda)^{-1} Av:v\in V},
\end{equation*}
a subspace of $D$. Let $W$ be the orthogonal complement of $V$ in $D$. The mapping $D\to D$ given by
\begin{equation*}
V\oplus W\ni (v\oplus w)\mapsto v+\pi_D F_{\Dom_F}(\lambda)^{-1} Av +w\in D
\end{equation*}
is invertible for every sufficiently negative $\lambda$ because $\|F_{\Dom_F}(\lambda)^{-1}\|\to0$ as $\lambda\to -\infty$. Its inverse tends to the identity as $\lambda\to-\infty$ and maps $V_\lambda$ to $V$. Let $S_\lambda:V_\lambda\to V$ be the restriction to $V_\lambda$ of this inverse and define $T_{\lambda,0}:V_\lambda\to D^\perp$ by 
\begin{equation*}
T_{\lambda,0}= \pi_{D^\perp} F_{\Dom_F}(\lambda)^{-1} AS_\lambda.
\end{equation*}
Then
\begin{equation*}
\set{v+T_{\lambda,0}v:v\in V_\lambda}=\set{v+F_{\Dom_F}(\lambda)^{-1}Av:v\in V} \subset K_\lambda,
\end{equation*}
therefore 
\begin{equation}\label{BecauseInK}
\big(v+T_{\lambda,0}v,A(v'+T_{\lambda,0}v')\big)_A=0\quad\text{ for every }v,v'\in V
\end{equation}
(cf. the proof of Lemma \ref{AllLambdaInSomeSpec}). Let $W_\lambda$ be the orthogonal complement of $V_\lambda$ in $D$. We now look for $T_{\lambda,1}:W_\lambda\to D^\perp$ such that with $T_\lambda:D\to D^\perp$ defined as $T_{\lambda,0}$ on $V_\lambda$ and as $T_{\lambda,1}$ on $W_\lambda$ we have that $\graph T_\lambda \in \SA$. Because of \eqref{SA} this will be the case iff for arbitrary $v,v'\in V_\lambda$ and $w,w'\in W_\lambda$ the quantity 
\begin{equation*}
\big(v+w+T_{\lambda,0}v+T_{\lambda,1}w,A(v'+w'+T_{\lambda,0}v'+T_{\lambda,1}w')\big)_A
\end{equation*}
vanishes. Using \eqref{BecauseInK} first and then several times that $D$ and $D^\perp$ are both in $\SA$ (so we can take advantage of \eqref{SA}) while keeping in mind that the ranges of $T_{\lambda,0}$ and $T_{\lambda,1}$ lie in $D^\perp$, the above expression is equivalent to 
\begin{multline*}
\big(v,AT_{\lambda,1}w'\big)_A+\big(T_{\lambda,0}v,Aw'\big)_A+\big(w,AT_{\lambda,0}v'\big)_A+\big(T_{\lambda,1}w,Av'\big)_A\\+\big(w,AT_{\lambda,1}w'\big)_A+\big(T_{\lambda,1}w,Aw'\big)_A
\end{multline*}
In order for this to vanish for all $V,v',w,w'$ it is necessary and sufficient that
\begin{equation*}
\big(v,AT_{\lambda,1}w'\big)_A+\big(T_{\lambda,0}v,Aw'\big)_A=0\text { and }\big(w,AT_{\lambda,1}w'\big)_A+\big(T_{\lambda,1}w,Aw'\big)_As
=0 
\end{equation*}
for all $v\in V_\lambda$ and $w,w'\in W_\lambda$. Letting $T_{\lambda,0}^*:D\to V_{\lambda,0}$ be the adjoint of $T_{\lambda,0}$, the first condition is equivalent to the requirement that $AT_{\lambda,1}=-T_{\lambda,0}^*A$, that is,
\begin{equation*}
T_{\lambda,1}=AT_{\lambda,0}^*A.
\end{equation*}
With this definition of $T_{\lambda,1}$ both $\big(w,AT_{\lambda,1}w'\big)_A$ and  $\big(T_{\lambda,1}w,Aw'\big)_A$ vanish because $W_\lambda \perp V_\lambda$ and $A$ is unitary. Thus $AT_{\lambda}:D\to D$ is selfadjoint, and since $T_\lambda\to 0$ as $\lambda\to -\infty$,
\begin{equation*}
D_\lambda=\graph T_\lambda \in \SA,\ K_\lambda\cap  D_\lambda\ne 0,\text{ and } D_\lambda\to D\text{ as }\lambda\to -\infty.
\end{equation*}
We have shown that $\mathfrak V_{D_F}$ consists of spectrally unstable domains.
\end{proof}

We end with an alternate argument to Proposition \ref{UnstableImpliesInVariety} that that all elements of $\SA\cap U_{D_F^\perp}$ are spectrally stable. Let $D_0\in \SA\cap U_{D_F^\perp}$ be arbitrary, let $T_0:D_F^\perp\to D_F$ be such that $D_0=\graph T_0$, let $S_0=AT_0$, and let $M>\|S_0\|$. Then
\begin{equation*}
U=\set{\graph T\,:\,T\in \L(D_F^\perp, D_F),\ S=AT\text{ selfadjoint, }\|S\|<M}
\end{equation*}
is a neighborhood of $D_0$ in $\SA$. There is $\zeta<0$ such that 
\begin{equation*}
(F_{\Dom_F}(\lambda)u,u)_A\leq -M\|u\|_A^2\quad \forall u\in D_F^\perp,\ \lambda<\zeta.
\end{equation*}
Let $D\in U$, so $D=\graph T$ with $S=AT:D_F^\perp\to D_F^\perp$ selfadjoint and $\|S\|<M$. Then
\begin{equation*}
((F_{\Dom_F}(\lambda)-S)u,u)_A\leq (-M+\|S\|)\|u\|_A^2\quad \forall u\in D_F^\perp,\ \lambda<\zeta
\end{equation*}
hence $\ker(F_{\Dom_F}(\lambda)-S)=0$ if $\lambda<\zeta$. Therefore
\begin{equation*}
 \spec(A_{\Dom_T})\subset [\zeta,\infty\neutral{(})
\end{equation*}
by \eqref{PartSpecD}.



\begin{thebibliography}{99}

\bibitem{BirSol87}
M.~Birman and M.~Z.~Solomjak, \emph{Spectral theory of selfadjoint operators in Hilbert space.} Translated from the 1980 Russian original by S. Khrushchëv and V. Peller. Mathematics and its Applications (Soviet Series). D. Reidel Publishing Co., Dordrecht, 1987.

\bibitem{Fthal1936}
H.~Freudenthal, \emph{\"Uber die Friedrichssche Fortsetzung halbbeschr\"ankter Hermitescher Operatoren}, Proc.~Akad.~Wet.~Amsterdam \textbf{39} (1936), 832--833.

\bibitem{Fr1934a}
K.~O.~Friedrichs, \emph{Spektraltheorie halbbeschr\"ankter Operatoren und Anwendung auf die Spektralzerlegung von Differentialoperatoren}, Math.~Ann.~\textbf{109} (1934), 465--487.

\bibitem{GiMe2003}
J.~B.~Gil and G.~A.~Mendoza, \emph{Adjoints of elliptic cone operators}, Amer. J. Math. \textbf{125} (2003), 357--408.

\bibitem{GiKrMe-b}
J.~B.~Gil, T.~Krainer, and G.~A.~Mendoza, {\em Resolvents of elliptic cone operators}, J. Funct. Anal.  {\bf 241}  (2006), 1--55.

\bibitem{GiKrMe-a}
J.~B.~Gil, T.~Krainer, and G.~A.~Mendoza, \emph{Geometry and spectra  of closed extensions of elliptic cone operators}, Canad. J. Math. \textbf{59} (2007), 742--794.

\bibitem{Lesch1997}
M.~Lesch, \emph{Operators of {F}uchs type, conical singularities, and asymptotic methods}, Teubner-Texte zur Math. vol 136, B.G. Teubner, Stuttgart, Leipzig, 1997. 

\bibitem{vNeumann}
J. v. Neumann, \emph{Allgemeine Eigenwerttheorie Hermitescher Funktionaloperatoren}, Math Ann.~\textbf{102} (1930), 49--131.

\bibitem{RandS}
M.~Reed and B.~Simon, \emph{Methods of modern mathematical physics. II. Fourier analysis, self-adjointness}, Academic Press [Harcourt Brace Jovanovich, Publishers], New York-London, 1975. 

\end{thebibliography}
\end{document}